\documentclass[reqno,12pt]{amsart}

\usepackage{amsmath}
\usepackage{amsfonts}
\usepackage{amssymb}
\usepackage{eufrak}
\usepackage{amscd}
\usepackage{amsthm}
\usepackage{amstext}
\usepackage[all]{xy}

 \newcommand{\K}{{\mathbb K}}

\newcommand{\id}{\operatorname{id}}

   \theoremstyle{plain}
   \newtheorem{thm}{Theorem}
   
   \newtheorem{lem}[thm]{Lemma}
   \newtheorem{cor}[thm]{Corollary}
   \theoremstyle{definition}

   \theoremstyle{remark}

 \title[]%
{One more pathology of $C^*$-algebraic tensor products}

\author{V. Manuilov}

\address{Department of Mechanics and Mathematics, Moscow State
University, Leninskie Gory, Moscow, 119991, Russia {\sl and\ }
Harbin Institute of Technology, Harbin, P. R. China}

\email{manuilov@mech.math.msu.su}

\thanks{The author acknowledges partial support from
RFFI, grant No. 07-01-00046}


\begin{document}

\begin{abstract}
We define a collection of tensor product norms for $C^*$-algebras
and show that a symmetric tensor product functor on the category
of separable $C^*$-algebras need not be associative.
\end{abstract}

\maketitle

\section{Introduction}

Following E. Kirchberg, \cite{Kirchberg-ICM}, we call a bifunctor
$(A,B)\to A\otimes_\alpha B$ a {\it $C^*$-algebraic tensor product
functor} if it is obtained by completing of the algebraic tensor
product $A\odot B$ of $C^*$-algebras in a functional way with
respect to a suitable $C^*$-norm $\|\cdot\|_\alpha$. We call such
a functor {\it symmetric} if the standard isomorphism $A\odot
B\cong B\odot A$ extends to an isomorphism $A\otimes_\alpha B\cong
B\otimes_\alpha A$. Similarly, we call it {\it associative} if the
standard isomorphism $A\odot(B\odot C)\cong (A\odot B)\odot C$
extends to an isomorphism $A\otimes_\alpha(B\otimes_\alpha C)\cong
(A\otimes_\alpha B)\otimes_\alpha C$ for any $C^*$-algebras $A$,
$B$, $C$. It is well known that both the minimal tensor product
functor $\otimes_{\min}$ and the maximal tensor product functor
$\otimes_{\max}$ are symmetric and associative.

In this paper we construct a collection of symmetric
$C^*$-algebraic tensor product functors related to asymptotic
homomorphisms of $C^*$-algebras. For technical reasons we restrict
ourselves to the category of {\it separable} $C^*$-algebras. Using
$C^*$-algebras related to property T groups \cite{WassT} we show
that some of these tensor product functors are not associative.

Recall that asymptotic homomorphisms of $C^*$-algebras were first
defined and studied in \cite{CH} in relation to topological
properties of $C^*$-algebras. The most important and the best
known case is the case of asymptotic homomorphisms from a
suspended $C^*$-algebra $SA$ to the $C^*$-algebra $\K$ of compact
operators, since the homotopy classes of those are the
$K$-homology of $A$, the $E$-theory. Asymptotic homomorphisms to
other $C^*$-algebras are less known. For example, it is known that
any asymptotic homomorphism to the Calkin algebra is homotopic to
a genuine homomorphism \cite{M-Crelle,MT4}. Even less is known
about asymptotic homomorphisms to $\mathbb B(H)$, where there is
no topological obstruction (recall that the $K$-groups of $\mathbb
B(H)$ are trivial). Such asymptotic homomorphisms are called {\it
asymptotic representations} and were first studied in relation to
the asymptotic tensor product $C^*$-algebras \cite{MT6} and to
semi-invertibility of $C^*$-algebra extensions \cite{MT7}.

\section{Definition of asymptotic $C^*$-tensor products}

Recall \cite{CH} that an asymptotic homomorphism $\varphi$ from a
$C^*$-algebra $A$ to a $C^*$-algebra $D$ is a family of maps
$\varphi=(\varphi_t)_{t\in[0,\infty)}:A\to D$ satisfying the
following properties:
\begin{itemize}
\item[1.]
the map $t\mapsto \varphi_t(a)$ is continuous for any $a\in A$;
\item[2.]
$\lim_{t\to\infty}\varphi_t(a+\lambda
b)-\varphi_t(a)-\lambda\varphi_t(b)=
\lim_{t\to\infty}\varphi_t(a^*)-\varphi_t(a)^*=
\lim_{t\to\infty}\varphi_t(ab)-\varphi_t(a)\varphi_t(b)=0$ for any
$a,b\in A$ and any $\lambda\in\mathbb C$.
\end{itemize}

Let $\mathbb L(H)$ be the algebra of bounded operators on a
separable Hilbert space $H$. Our point is that we would like to
consider $D$ as a $C^*$-subalgebra of $\mathbb L(H)$:
$D\subset\mathbb L(H)$. We also view asymptotic homomorphisms to
$D$ as asymptotic representations on $H$ taking values in $D$. We
are mostly interested in the special case $D=\mathbb
K^\infty=\prod_{n=1}^\infty\mathbb K=\prod_{n=1}^\infty\mathbb
K(H_n)$, where $H_n=H$ for all $n\in\mathbb N$ and $\mathbb
K=\mathbb K(H)$ is the $C^*$-algebra of compact operators on $H$.

Let $A$, $B$ be separable $C^*$-algebras and let
$\varphi=(\varphi_t)_{t\in[0,\infty)},\psi=(\psi_t)_{t\in[0,\infty)}$
be asymptotic representations of $A$ and $B$ respectively, taking
values in $D$.

Let $A \odot B$ be the algebraic tensor product of $A$ and $B$.
For each $a \in A$ and $b \in B$, we can define elements
$a^{\varphi \otimes \psi} , b^{\varphi \otimes \psi} \in C_b\left(
[0,\infty), \mathbb L(H \otimes H)\right)$ by
$$
a^{\varphi \otimes \psi}(t) = \varphi_t(a) \otimes 1_H \quad {\rm
and}\quad b^{\varphi \otimes \psi}(t) = 1_H \otimes \psi_t(b).
$$

Note that $a^{\varphi \otimes \psi}(t)\cdot b^{\varphi \otimes
\psi}(t)\in C_b\left([0,\infty), D\otimes_{\min}D)\right)$, where
$\otimes_{\min}$ denotes the minimal tensor product of
$C^*$-algebras.

We can then define a $*$-homomorphism
$$
\varphi \otimes \psi : A \odot B \to C_b\left( [0,\infty),
D\otimes_{\min}D\right)/ C_0\left( [0,\infty),
D\otimes_{\min}D\right)
$$
such that
$$
\varphi \otimes \psi\Bigl( \sum\nolimits_i a_i \otimes b_i\Bigr) =
q\Bigl(\sum\nolimits_i a_i^{\varphi \otimes \psi}\cdot
b_i^{\varphi \otimes \psi}\Bigr),
$$
where
 $$
q:C_b\left( [0,\infty), D\otimes_{\min}D\right)\to C_b\left(
[0,\infty), D\otimes_{\min}D\right)/C_0\left( [0,\infty),
D\otimes_{\min}D\right)
 $$
is the quotient map. Note that
 $$
\|\varphi\otimes\psi(c)\|=\lim\sup\nolimits_{t\to\infty}\Bigl\|\sum\nolimits_{i}\varphi(a_i)\otimes\psi(b_i)\Bigr\|
 $$
for any $c=\sum_i a_i \otimes b_i\in A\odot B$. We can now define
a seminorm $\| \cdot \|_{D,0}$ on $A \odot B$ by
$$
\| c \|_{D,0} = \sup\nolimits_{\varphi, \psi} \|\varphi
\otimes\psi(c)\|,
$$
where we take the supremum over all pairs $(\varphi,\psi)$, where
$\varphi$ and $\psi$ are asymptotic representations of $A$ and
$B$, respectively, taking values in $D$.

Note that a genuine $*$-homomorphism from $A$ to $D$ can be
considered as an asymptotic representation in the obvious way. So,
if $D=\mathbb L(H)$ then $\|\cdot\|_{D,0}\geq\|\cdot\|_{\min}$,
and $\|\cdot\|_{D,0}$ is a norm. This norm coincides with the
symmetric asymptotic tensor norm defined in \cite{MT7}. More
generally, the seminorm $\|\cdot\|_{D,0}$ is a norm if there exist
{\it faithful} asymptotic representations of $A$ and $B$ taking
values in $D$. Remark that there are other $C^*$-algebras $D$,
besides $\mathbb L(H)$, that admit faithful asymptotic
representations of any separable $C^*$-algebra. For example, it
follows from \cite{MT4} that one can take the {\it coarse Roe
algebra} of $\mathbb Z$ as $D$.

In general, the seminorm $\| \cdot \|_{D,0}$ may be degenerate
(e.g. it may happen that any asymptotic representation of a
$C^*$-algebra $A$ taking values in some $D$ may be asymptotically
equivalent to zero, see Lemma \ref{degenerate} below), so let us
define the norm $\|\cdot\|_D$ on $A\odot B$ by
 $$
\|c\|_D=\max\{\|c\|_{\min},\|c\|_{D,0}\},
 $$
where $c\in A\odot B$. Clearly, $\|\cdot\|_D$ is a $C^*$-norm,
hence a cross-norm, and
$$
\|\cdot\|_{\min}\leq\|\cdot\|_{D}\leq\|\cdot\|_{\max}.
$$
We denote by $A\otimes_{D}B$ the $C^*$-algebra obtained by
completing $A\odot B$ with respect to the norm $\|\cdot\|_{D}$.
Obviously the correspondence $(A,B)\mapsto A\otimes_D B$ is a
$C^*$-algebraic tensor product functor on the category of
separable $C^*$-algebras.

 \begin{lem}
The functor $\otimes_D$ is symmetric.

 \end{lem}
 \begin{proof}
Obvious.

 \end{proof}

\section{Asymptotic representations taking values in $\mathbb K^\infty$}

Let $G$ be a residually finite infinite property T group, let
$\pi_n$ be the sequence of all non-equivalent irreducible unitary
representations on finitedimensional Hilbert spaces $H_n$ and let
$\overline{A}$ be the $C^*$-algebra generated by operators
$\oplus_{n=1}^\infty\pi_n(g)$, $g\in G$. We denote by $E$ the
$C^*$-subalgebra in $\mathbb L(\oplus_{n=1}^\infty H_n)$ generated
by $\overline{A}$ and by compact operators:
$E=\overline{A}+\mathbb K$. Put $A=E/\mathbb K$. This
$C^*$-algebra was first considered by S. Wassermann and we refer
to his paper \cite{WassT} for more details.

 \begin{lem}\label{degenerate}
Let $\varphi=(\varphi_t)_{t\in[0,\infty)}:A\to \mathbb K^\infty$
be an asymptotic homomorphism. Then $\varphi$ is asymptotically
equivalent to zero, i.e. $\lim_{t\to\infty}\varphi_t(a)=0$ for any
$a\in A$.

 \end{lem}
 \begin{proof}
Let $q_n:\prod_{n=1}^\infty\mathbb K\to \mathbb K$ is the
projection onto the $n$-th copy. Then, for any $\varepsilon>0$
there exists $t_0$ and continuous families $p_n(t)\in\mathbb K_n$
of finitedimensional projections such that
 \begin{equation}\label{e1}
\|q_n\circ\varphi_t(1)-p_n(t)\|<\varepsilon
 \end{equation}
for any $t>t_0$ and any $n\in\mathbb N$. Let $N_n$ denote the rank
of $p_n(t)$. Since all projections of the same finite rank are
unitarily equivalent, without loss of generality (by changing
$\varphi$ by a unitarily equivalent asymptotic homomorphism) we
may assume that all $p_n(t)$ are $t$-independent, $p_n(t)=p_n$.
Then the formula $\psi_t(a)=p_n(q_n\circ\varphi_t(a))p_n$ defines
an asymptotic homomorphism from $A$ to the matrix algebra. The
group $G$ with the stated properties is known to be finitely
generated, so without loss of generality we may assume that
$\psi_t(g_i)$ are unitaries, where $g_i\in G$, $i=1,\ldots,k$, are
generators for $G$.

Since the direct product of $k$ copies of the unitary group
$U_{N_n}$ is compact, so the set
$\{(\psi_t(g_1),\ldots,\psi_t(g_k)):t\in[0,\infty)\}$ has an
accumulation point $(u_1,\ldots,u_k)\in U_{N_n}^k$. If we put
$\sigma(g_i)=u_i$ then this map extends to a genuine
representation of $G$ of dimension $N_n$. Indeed, $G$ is a
quotient of the free group $\mathbb F_k$ generated by
$g_1,\ldots,g_k$ modulo some relations and each $\psi_t$ and
$\sigma$ obviously define representations of $\mathbb F_k$, which
we denote by the same characters. If $r\in\mathbb F_k$ is a
relation then $\lim_{t\to\infty}\|\psi_t(r)-p_n\|=0$. Therefore,
$\sigma(r)=p_n$, hence $\sigma$ factorizes through a
representation of $G$.

Suppose that $p_n\neq 0$ for some $n$. This implies that the
representation $\sigma$ is non-zero, hence it equals one of
$\pi_j$. Since
 $$
\|a\|\geq\lim\sup\nolimits_{t\to\infty}\|\psi_t(a)\|\geq\|\sigma(a)\|=\|\pi_j(a)\|,
 $$
we have $\lim\sup_{n\to\infty}\|\pi_n(a)\|\geq\|\pi_j(a)\|$ for
any $a\in A$. Then the identity map of $G$ extends to a
$*$-homomorphism $i:A\to C^*_{\pi_j}(G)$, where $C^*_{\pi}(G)$
denotes the $C^*$-algebra generated by the representation $\pi$.
Tensoring it by $\id_{C^*_{\overline{\pi}_j}(G)}$, where
$\overline{\pi}$ denotes the contragredient representation for a
representation $\pi$, we get a $*$-homomorphism
 \begin{equation}\label{map1}
i\otimes \id_{{C^*_{\overline{\pi}_j}(G)}}:A\otimes
C^*_{\overline{\pi}_j}(G)\to C^*_{\pi_j}(G)\otimes
C^*_{\overline{\pi}_j}(G).
 \end{equation}
We do not specify the tensor product norm here because
$C^*_{\overline{\pi}_j}(G)$ is finitedimensional, hence nuclear.
It was shown in \cite{WassT} that the norm on the left hand side
of (\ref{map1}) is strictly smaller than the norm on the right
hand side, so this $*$-homomorphism cannot exist. This
contradiction shows that $p_n=0$ for all $n$, hence (\ref{e1})
implies that $\lim_{t\to\infty}\|q_n\circ\varphi_t(1)\|=0$
uniformly in $n$, therefore,
$\lim_{t\to\infty}\|q_n\circ\varphi_t(a)\|=0$ uniformly in $n$ for
any $a\in A$.

 \end{proof}

 \begin{cor}\label{min}
For $A$ defined above, one has $A\otimes_{\mathbb K^\infty}
B=A\otimes_{\min} B$ for any $C^*$-algebra $B$.

 \end{cor}

 \begin{proof}
Since
 $$
\|\varphi\otimes\psi(a\otimes
b)\|=\lim\sup\nolimits_{t\to\infty}\|\varphi_t(a)\otimes\psi_t(b)\|=
\lim\sup\nolimits_{t\to\infty}\|\varphi_t(a)\|\cdot\|\psi_t(b)\|=0
 $$
for any $a\in A$, $b\in B$ and for any asymptotic representations
$\varphi$ and $\psi$, one has
 $$
\|a\otimes b\|_{\mathbb
K^\infty,0}=\sup\nolimits_{\varphi,\psi}\|\varphi\otimes\psi(a\otimes
b)\|=0,
 $$
hence $\|c\|_{\mathbb K^\infty,0}=0$ for any $c\in A\odot B$,
therefore,
 $$
\|c\|_{\mathbb K^\infty}=\max\{\|c\|_{\min},0\}=\|c\|_{\min}.
 $$

 \end{proof}

\section{An example of an asymptotic representation taking values in $\mathbb K^\infty$}

Let $C=C_0(0,1]$. We are going to construct an asymptotic
representation $\varphi^0$ of $C\otimes A$ taking values in
$\mathbb K^\infty$. (We do not specify here the tensor norm since
$C$ is nuclear.)

Let $\chi:A\to E$ be a continuous homogeneous selfadjoint
selection map, cf. \cite{BG}. We denote by $P_n$ the projection in
$\oplus_{n=1}^\infty H_n$ onto $H_n$. For $a\in A$ put
$\alpha(a)=\iota\circ\chi(a)$, where $\iota:E\to\mathbb
L(\oplus_{n=1}^\infty H_n)$ is the standard inclusion.


Let $\{\tau_n\}_{n\in\mathbb N}$ be a dense sequence of points in
$(0,1)$.

For $t=k\in\mathbb N$ and for $f\in C$ put
 $$
\beta_k(f)=\sum_{n=k+1}^\infty f(\tau_n)P_n.
 $$
(this sum also is convergent with respect to the $*$-strong
topology). If $k<t<k+1$ then put
 $$
\beta_t(f)=f((t-k)\tau_{k+1})P_{k+1}+\beta_{k+1}(f).
 $$

Let $F\in C\otimes A$. One can consider $F$ as a continuous
function on $[0,1]$ taking values in $A$ such that $F(0)=0$. Put
 $$
\phi_k(F)=\sum_{n=k+1}^\infty P_n\alpha(F(\tau_n))P_n,
 $$
where the sum is $*$-strongly convergent, and
 $$
\phi_t(F)=P_{k+1}\alpha(F((t-k)\tau_{k+1}))P_{k+1}+\phi_k(F)
 $$
for $k<t<k+1$.

 \begin{lem}
The family of maps $(\phi_t)_{t\in[0,\infty)}$ is an asymptotic
representation of $C\otimes A$ taking values in
$\prod_{n=1}^\infty\mathbb L(H_n)\subset\mathbb K^\infty$.

 \end{lem}

 \begin{proof}
By the definition, the maps $\phi_t$, $t\in[0,\infty)$, take
values in $\prod_{n=1}^\infty\mathbb L(H_n)$, so we only need to
check that algebraic properties hold asymptotically. Let us check
that for multiplication, as other properties can be checked in the
same way. Let $F_1,F_2\in C\otimes A$, then the operators
 $$
K_n=\alpha(F_1(\tau_n)F_2(\tau_n))-\alpha(F_1(\tau_n))\alpha(F_2(\tau_n))\in\mathbb
K
 $$
lie in a compact subset of $\mathbb K$ (the image of $[0,1]$ under
the continuous map
$\tau\mapsto\alpha(F_1(\tau)F_2(\tau))-\alpha(F_1(\tau))\alpha(F_2(\tau))$),
hence
 $$
\lim_{k\to\infty}\phi_k(F_1F_2)-\phi_k(F_1)\phi_k(F_2)=\lim_{k\to\infty}\sum_{n=k+1}^\infty
P_nK_nP_n=0.
 $$
Finally, we easily pass to the continuous parameter:
$\lim_{t\to\infty}\phi_t(F_1F_2)-\phi_t(F_1)\phi_t(F_2)=0$.

 \end{proof}

Note that if $F=f\otimes a\in C\otimes A$ then
 $$
\phi_t(f\otimes a)=\sum_{n=k}^\infty P_n\alpha(a)P_n\cdot
\beta_t(f).
 $$

\section{Compairing tensor norms}

Let $B$ be the $C^*$-algebra generated by operators
$\oplus_{n=1}^\infty\overline{\pi}_n(g)$, $g\in G$, where
$\overline{\pi}$ denotes the contragredient representation for
$\pi$.

 \begin{thm}
The tensor products $C\otimes_{\mathbb K^\infty}(A\otimes_{\mathbb
K^\infty}B)$ and $(C\otimes_{\mathbb K^\infty}A)\otimes_{\mathbb
K^\infty}B$ are not canonically isomorphic.

 \end{thm}

 \begin{proof}
Let $f\in C$ be the identity function, $f(\tau)=\tau$, and let
$\{g_1,\ldots,g_m\}$ be a symmetric set of generators of the group
$G$ as above. We identify the group elements with the
corresponding unitaries in $C^*$-algebras generated by
representations of $G$ (like $B$) and in their quotients (like
$A$). Let $d=\sum_{i=1}^m f\otimes g_i\otimes g_i\in C\odot A\odot
B$. Denote by $\|\cdot\|_1$ and by $\|\cdot\|_2$ the norms on
$C\odot A\odot B$ inherited from $C\otimes_{\mathbb
K^\infty}(A\otimes_{\mathbb K^\infty}B)$ and $(C\otimes_{\mathbb
K^\infty}A)\otimes_{\mathbb K^\infty}B$ respectively. Our aim is
to show that $\|d\|_1\neq\|d\|_2$.

It follows from Lemma \ref{min} and from amenability of $C$ that
$C\otimes_{\mathbb K^\infty}(A\otimes_{\mathbb
K^\infty}B)=C\otimes_{\min}(A\otimes_{\min}B)$, so
 $$
\|d\|_1=\|f\|\cdot\Bigl\|\sum\nolimits_{i=1}^mg_i\otimes
g_i\Bigr\|_{\min}=\Bigl\|\sum\nolimits_{i=1}^mg_i\otimes
g_i\Bigr\|_{\min}.
 $$
It was shown in \cite{WassT} that the latter norm is strictly
smaller than $m$, so
 \begin{equation}\label{<m}
\|d\|_1<m.
 \end{equation}

When estimating the norm $\|\cdot\|_2$ from below, we may use two
special asymptotic representations instead of taking the supremum
over all of them. Let us take $\phi_t$ for $C\otimes A$ and the
identity representation for $B$. Then
 \begin{eqnarray*}
\|d\|_2&\geq&\lim\sup\nolimits_{t\to\infty}
\Bigl\|\sum\nolimits_{i=1}^m\phi_t(f\otimes
g_i)\otimes\sum\nolimits_{n=1}^\infty\overline{\pi}_n(g_i)P_n\Bigr\|\\
&=&\lim\sup\nolimits_{t\to\infty}\Bigl\|\sum\nolimits_{i=1}^m\sum\nolimits_{n=1}^\infty
P_n\beta_t(f)\alpha(g_i)P_n\otimes
\sum\nolimits_{n=1}^\infty\overline{\pi}_n(g_i)P_n\Bigr\|\\
&\geq&\lim\sup\nolimits_{t\to\infty}\sup\nolimits_n\Bigl\|P_n\beta_t(f)\sum\nolimits_{i=1}^m
\alpha(g_i)P_n\otimes
\overline{\pi}_n(g_i)P_n\Bigr\|\\
&\geq&\lim\sup\nolimits_{n\to\infty}\Bigl\|P_n\beta_n(f)\sum\nolimits_{i=1}^m
\alpha(g_i)P_n\otimes
\overline{\pi}_n(g_i)P_n\Bigr\|\\
&=&\lim\sup\nolimits_{n\to\infty}f(\tau_n)\cdot\Bigl\|\sum\nolimits_{i=1}^m
P_n\alpha(g_i)P_n\otimes \overline{\pi}_n(g_i)P_n\Bigr\|\\
&\geq&\lim\sup\nolimits_{j\to\infty}f(\tau_{n_j})\cdot\Bigl\|\sum\nolimits_{i=1}^m
P_{n_j}\alpha(g_i)P_{n_j}\otimes\overline{\pi}_{n_j}(g_i)\Bigr\|,
 \end{eqnarray*}
where $\{n_j\}$ is any increasing subsequence of integers. Since
the sequence $\{\tau_n\}_{n=1}^\infty$ is dense in $[0,1]$, we can
find a subsequence $\{n_j\}_{j=1}^\infty$ such that
$\lim_{j\to\infty}\tau_{n_j}=1$. Then
 \begin{eqnarray*}
\|d\|_2&\geq&\lim\sup\nolimits_{j\to\infty}f(\tau_{n_j})\cdot\Bigl\|\sum\nolimits_{i=1}^m
P_{n_j}\alpha(g_i)P_{n_j}\otimes\overline{\pi}_{n_j}(g_i)\Bigr\|\\
&=&\lim\sup\nolimits_{j\to\infty}\Bigl\|\sum\nolimits_{i=1}^m
P_{n_j}\alpha(g_i)P_{n_j}\otimes\overline{\pi}_{n_j}(g_i)\Bigr\|\\
&=&\lim\sup\nolimits_{j\to\infty}\Bigl\|\sum\nolimits_{i=1}^m
\pi_{n_j}(g_i)\otimes\overline{\pi}_{n_j}(g_i)\Bigr\|\\
&=&\Bigl\|\sum\nolimits_{i=1}^m
\pi_{n_j}(g_i)\otimes\overline{\pi}_{n_j}(g_i)\Bigr\| \ =\
\sum\nolimits_{i=1}^m 1\ =\ m.
 \end{eqnarray*}
On the other hand, $\|d\|_2\leq \sum_{i=1}^m\|f\otimes g_i\otimes
g_i\|_2=m$, so we have
 \begin{equation}\label{=}
\|d\|_2=m.
 \end{equation}

Compairing (\ref{<m}) and (\ref{=}), we conclude that these two
norms are different.

 \end{proof}

\end{document}